\documentclass[10pt]{article}

\usepackage[a4paper,
            bindingoffset=0.2in,
            left=0.8in,
            right=0.8in,
            top=0.8in,
            bottom=0.8in,
            footskip=.25in]{geometry}
\usepackage{xcolor,placeins}
\usepackage[ruled,vlined]{algorithm2e}
\usepackage{tikz}

\usetikzlibrary{calc,trees,positioning,arrows,chains,shapes.geometric,  decorations.pathreplacing,decorations.pathmorphing,shapes,  matrix,shapes.symbols}
\usepackage{tikz-cd}
\tikzcdset{scale cd/.style={every label/.append style={scale=#1},
    cells={nodes={scale=#1}}}}

\usepackage{hyperref} 
\usepackage{multirow,multicol}

\usepackage{latexsym, bm, amsmath, amssymb, graphics, amsthm,
enumerate} 
\usetikzlibrary{arrows.meta,
                positioning,
                shapes}
\newtheorem{theorem}{Theorem}

\newtheorem{proposition}{Proposition}

\newtheorem{corollary}{Corollary}

\newtheorem{Def}{Definition}

\newtheorem{Ex}{Example}

\newcommand{\X}{X}

\renewcommand{\Re}{\ensuremath{\mathbb{R}}}

\usepackage{comment}

\begin{document}

\title{A Conditional Probability Hierarchy for Stochastic Choice\footnote{We thank Samson Abramsky, Victor Aguiar, James Bergin, Inacio Bo, Stefan Bucher, Emiliano Catonini, Vered Kurtz-David, Igor Kopylov, Kenway Louie, Jay Lu, Gerelt Tserenjigmid, and Stuart Zoble for valuable input, and Kai Steverson for early inspiration for this project.  The editor and referees provided very important suggestions for improvement.  We are especially indebted to a referee for suggesting the definition of and characterization by Conditional Consistency.  We also thank participants in workshops at University of Macau, the Harbin Institute of Technology (Shenzhen), NYU Shanghai, the 2025 Barcelona School of Economics Summer Forum, the Society for the Advancement of Economic Theory 2025 Conference (SAET 2025), the 14th International Symposium on Imprecise Probabilities: Theories and Applications (ISIPTA 2025), and the 13th World Congress of the Econometric Society (ESWC 2025) for comments.  Brandenburger acknowledges financial support from NYU Stern School of Business, NYU Shanghai, and J.P. Valles.  Yang acknowledges financial support from National Natural Science Foundation of China (NSFC), China, Award No.~72403259.}}

\author
{Erya Yang \footnote{Lingnan College, Laboratory of Mezzoeconomics and Regional Industrial Coordinated Development, Shenzhen Institute of Economics of Lingnan College, Sun Yat-sen University, China; yangery@mail.sysu.edu.cn; corresponding author}
\and
{Adam Brandenburger \footnote{Stern School of Business, Tandon School of Engineering, NYU Shanghai, New York University, New York, NY 10012, U.S.A., adam.brandenburger@stern.nyu.edu}}
    }
\date{September 29, 2026}
\maketitle

\thispagestyle{empty}

\begin{abstract}
We introduce point conditional probability spaces (PCPSs) as primitive building blocks for stochastic choice.  This concept goes back to R\'enyi (1955), who proposed conditional probability spaces (CPSs) as a basis for probability theory.  Luce (1959) noted the connection between CPSs and stochastic choice, and Cerreia-Vioglio et al.~(2021) have developed the connection further.  A PCPS is a CPS each of whose component probability measures concentrates on a singleton selection.  We build a four-level PCPS-based hierarchy of families of stochastic choice rules.  Level 1 consists of PCPSs, Level 2 is made up of finitely additive ``conditionally consistent'' mixtures of PCPSs, Level 3 comprises all finitely additive probabilistic mixtures of PCPSs, and Level 4 consists of all finitely additive signed probabilistic mixtures of PCPSs.  We construct our hierarchy at a general measure-theoretic level that encompasses infinite choice sets.  We also connect each level of our hierarchy to well-known axioms for stochastic choice, namely, the Weak Axiom of Stochastic Revealed Preference, Independence of Irrelevant Alternatives, and no Dutch book.  We establish the relationship between total orders and PCPSs and demonstrate a sense in which PCPSs can be a more parsimonious representation of choice.
\end{abstract}

\section{Introduction}
The field of stochastic choice has deep roots in psychology (Luce, 1959), discrete choice (McFadden, 1974), and behavioral decision theory (Gabaix, 2019).  The starting point of theoretical development in this field is the classic Independence of Irrelevant Alternatives Axiom (IIA) due to Luce (1959).  But, very quickly, the empirical validity of this axiom was called into question (Debreu, 1960) and a vast and rich literature has since been built on the basis of alternative axioms and choice rules.  Strzalecki (2022) and Davis-Stober et al.~(2023) are recent surveys of the theoretical and empirical literature.

In this paper, we go back to the beginning and IIA, and we make this the starting point for a new architecture of stochastic choice.  Our starting observation is that a stochastic choice rule satisfying IIA is formally equivalent to the probability-theory concept of a conditional probability space (CPS) due to R\'enyi (1955, 1956).  This equivalence was briefly noted by Luce (1959) for strictly positive choice probabilities and has been generalized by Cerreia-Vioglio et al.~(2021).  A CPS is a family of probability measures indexed by a family of observable events.  The key condition on a CPS is a chain rule, which disciplines the way probabilities are updated on nested sequences of events.  Clearly, a CPS can also carry an interpretation as a stochastic choice rule.  The observable events become choice sets and the associated probabilities become choice probabilities.

To build our architecture, we start with what we call a point CPS (PCPS), which is a CPS based on the selection of a single point for each conditioning set (i.e., each component measure is Dirac).  This is the bottom level of our hierarchy.  Level 2 of our hierarchy is made up of mixtures of PCPSs that satisfy a conditional consistency axiom, which is an analog, for the mixing operation, to the chain rule for CPSs.  Level 3 of our hierarchy comprises all probabilistic mixtures of PCPSs, and Level 4 consists of all signed probabilistic mixtures of PCPSs.  Figure \ref{fig1} summarizes our PCPS-based hierarchy.
\vspace{0.1in}

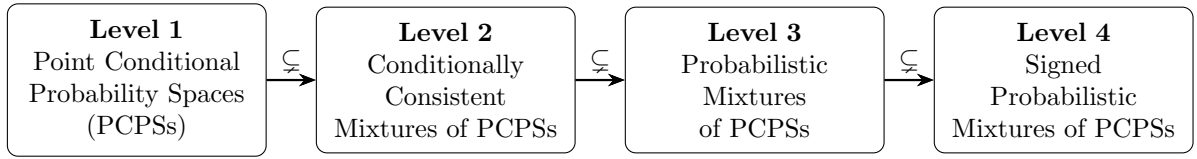
\begin{figure}[h]
\centering
\begin{tikzpicture}[
levelbox/.style={
draw,
rounded corners,
align=center,
inner sep=6pt,
minimum height=1.35cm,
text width=3.0cm
},
arrow/.style={
-{Stealth[length=2.2mm]},
thick
},
node distance=0.65cm
]

\node[levelbox] (L1) {
\textbf{Level 1}\\
Point Conditional\\
Probability Spaces\\
(PCPSs)
};

\node[levelbox, right=of L1] (L2) {
\textbf{Level 2}\\
Conditionally Consistent\\
Mixtures of PCPSs
};

\node[levelbox, right=of L2] (L3) {
\textbf{Level 3}\\
Probabilistic Mixtures\\
of PCPSs
};

\node[levelbox, right=of L3] (L4) {
\textbf{Level 4}\\
Signed \\
Probabilistic \\
Mixtures of PCPSs
};

\draw[arrow] (L1) -- (L2);
\draw[arrow] (L2) -- (L3);
\draw[arrow] (L3) -- (L4);

\draw[arrow] (L1) -- node[above] {$\subsetneq$} (L2);
\draw[arrow] (L2) -- node[above] {$\subsetneq$} (L3);
\draw[arrow] (L3) -- node[above] {$\subsetneq$} (L4);

\end{tikzpicture}

\caption{A Conditional Probability Hierarchy}
\label{fig1}
\end{figure}

We build our hierarchy at a general measure-theoretic level, where the underlying set of elementary choice objects is a measurable space and the family of choice sets is an arbitrary family of nonempty measurable sets (for some results) or is closed under nonempty finite intersections (for other results).  No further closure conditions are needed, which means that we can represent very general patterns of behavior.  Hildenbrand (1971) is an early study of stochastic choice on infinite sets in the setting of general equilibrium theory with random preferences.  McFadden and Richter (1971, 1991) and McFadden (2005, Theorem 5.3) obtain a characterization of random utility in the case where the underlying space is Polish and the choice sets are compact.  Clark (1996) is a pioneering general characterization result for random utility based on the de Finetti no Dutch book condition (de Finetti, 1937, 1974).  (One of our results is closely inspired by this work.)  Cohen (1980) proves an equivalence between random choice and random utility in the finitary case where the set of elementary choice objects is infinite while each choice set is finite.  Our architecture is intrinsically infinite in that we allow choice sets to be infinite.

Each level of our hierarchy can be related to existing axioms or analyses of stochastic choice.  For Level 1, we prove that, if the conditioning family is closed under nonempty finite intersections, the PCPS concept is equivalent to point stochastic choice rules satisfying the Weak Axiom of Stochastic Revealed Preference (Bandyopadhyay, Dasgupta, and Pattanaik, 1999, 2002; Dasgupta and Pattanaik, 2007).  For Level 2, we show that, for an arbitrary conditioning family, a probabilistic mixture of PCPSs is conditionally consistent if and only if its associated stochastic choice rule is a CPS.  In turn, if the conditioning family is closed under nonempty finite intersections, CPSs are exactly the stochastic choice rules satisfying a Generalized Independence of Irrelevant Alternatives.  For Level 3, we prove that, for an arbitrary conditioning family, the family of (all) probabilistic mixtures of PCPSs coincides with the family of stochastic choice rules that satisfy a de Finetti-like no Dutch book condition (de Finetti, 1937, 1974).  Finally, for Level 4, if the underlying set of choices is finite and the conditioning family consists of all nonempty subsets, it follows from Dogan and Yildiz  (2023, Theorem 2) that this level imposes no restriction on a stochastic choice rule -- so that it is the top of our hierarchy.  A characterization of Level 4 when the underlying choice set is infinite appears to be open.
 
A natural question concerns the relationship between our hierarchy and one built on total (linear) orders over choices.  The short answer is that, in our setting, it is more parsimonious to work with PCPSs vs.~total orders.  To understand this, first note that there is an obvious map from total orders to PCPSs that selects the point in each choice set that is top-ranked under the total order.  However, this map is not necessarily surjective, which says that PCPSs are genuinely different from total orders as a building block for stochastic choice.  Under an additional closure condition on the family of choice sets (stronger than we make elsewhere), we obtain a surjection but not necessarily an injection.  This says that, even in this case, there is parsimony in working with PCPSs.  To be precise, this is if we first specify the family of choice sets.  Then, a total order may contain redundant information about counterfactual behavior (as reflected by lack of injectivity).  If the philosophy is that preference precedes the specification of choice sets, then our argument loses weight and total orders become a very natural starting point.  Separately, we repeat that, at a broader level, our goal in this paper is to articulate the connection (going back to Luce, 1959) between stochastic choice rules and CPSs.  For this reason, the paper is couched in terms of PCPSs and notions of mixing over PCPSs.

The organization of the rest of the paper is as follows.  Section~\ref{sec:2} introduces stochastic choice rules and CPSs in our setting, and it establishes the relationship to the Weak Axiom of Stochastic Revealed Preference.  Section~\ref{sec:3} defines Level 2 of our hierarchy in terms of Conditionally Consistent mixtures of PCPSs and proves results about this case, including the connection to Generalized Independence of Irrelevant Alternatives.  Section~\ref{sec:4} treats all probabilistic mixtures of PCPSs and proves the equivalence to stochastic choice rules satisfying no Dutch book.  Section~\ref{sec:5} defines Level 4 of our hierarchy as signed probabilistic mixtures of PCPSs, which constitutes the top of our hierarchy in the finite case.  Section~\ref{sec:6} establishes the surjection and bijection results for total orders and PCPSs.  Section~\ref{sec:7} notes the strict nesting of the levels of our hierarchy, points out an illuminating non-commutativity with CPSs, discusses some literature that adds finer detail relative to our hierarchy, and suggests some future directions.  Proofs of most results are in the main text, while a longer proof and a set of examples can be found in two appendices.

\section{Stochastic Choice Rules and CPSs} \label{sec:2}
In this section, we define the basic objects with which we work throughout.  Given a measurable space $(\Omega, \mathcal{F})$, we write $\Delta(\Omega)$ for the set of all probability measures on $(\Omega, \mathcal{F})$.  Let $\cal G$ be a subfamily of $\cal F$.  We will always assume that $\emptyset \notin \cal G$.

\begin{Def} \label{SCR}
A \textbf{stochastic choice rule} (relative to $\cal G$) is a map $c : {\cal G} \rightarrow \Delta(\Omega)$, which we will also write as $G \mapsto c_G(\cdot)$ for $G \in {\cal G}$, satisfying:
\begin{enumerate}
\item $c_G(G) = 1$ for every $G \in \cal G$.
\end{enumerate} 
\end{Def}

\begin{Def} \label{CPS}
A \textbf{conditional probability space} (\textbf{CPS}) (relative to $\cal G$) is a map $p : {\cal G} \rightarrow \Delta(\Omega)$, which we will also write as $G \mapsto p_G(\cdot)$ for $G \in {\cal G}$, satisfying:
\begin{enumerate}
\item $p_G(G) = 1$ for every $G \in \cal G$;
\item $p_G(E) = p_G(F) \times p_F(E)$ for every $E \subseteq F \subseteq G$ with $E \in \cal F$ and $F, G \in \cal G$.
\end{enumerate}
\end{Def}

We see that a CPS is a stochastic choice rule satisfying an extra chain rule requirement, namely, Condition 2 of Definition \ref{CPS}.  (In R\'enyi 1955, Condition 2 takes the form: If $E, F \in \cal F$, $G \in \cal G$, and $F \cap G \in \cal G$, then $p_G(E \cap F) = p_G(F) \times p_{F \cap G}(E)$. This is readily seen to be equivalent to our Condition 2.)

We now define the cases of point stochastic choice rules and point CPSs.  Let $s: {\cal G} \rightarrow \Omega$ be a selection on $\cal G$, that is, for each $G \in \cal G$, $s(G) \in G$.

\begin{Def} \label{PSCR}
Fix a stochastic choice rule $c: {\cal G} \to \Delta(\Omega)$. Suppose there is a selection $s$ such that, for each $G \in \cal G$ and $F \in \cal F$: 
\begin{equation}~\label{eq7}
c_G(F)=\mathbf{1}_F(s(G)).
\end{equation}
That is, each $c_G(\cdot)$ is the Dirac measure concentrated on $s(G)$.  We call such a stochastic choice rule a \textbf{point stochastic choice rule}.
\end{Def}

\begin{Def} \label{PCPS}
Fix a CPS $p : {\cal G} \rightarrow \Delta(\Omega)$. Suppose there is a selection $s$ such that, for each $G \in \cal G$ and $F \in \cal F$:
\begin{equation}~\label{eq6}
p_G(F) = \mathbf{1}_F(s(G)).
\end{equation}
That is, each $p_G(\cdot)$ is the Dirac measure concentrated on $s(G)$. We call such a CPS a \textbf{point CPS} (\textbf{PCPS}). 
\end{Def}

We next establish a characterization of Level 1 of our conditional probability hierarchy.  We make use of a stochastic analog to the Weak Axiom of Revealed Preference (WARP) of classic demand theory, introduced by Bandyopadhyay, Dasgupta, and Pattanaik (1999, 2002) and Dasgupta and Pattanaik (2007).

\begin{Def} \label{WASRP}
A stochastic choice rule $G \mapsto c_G(\cdot)$ satisfies the \textbf{Weak Axiom of Stochastic Revealed Preference} (\textbf{WASRP}) if for every $G, H \in \cal G$ and every $F \in \cal F$ with $F \subseteq G \cap H$:
\begin{equation}~\label{eq8}
c_H(F) - c_G(F) \leq c_G(G\backslash H).
\end{equation}
\end{Def}

To explain WASRP, observe that the axiom does not require equality of the choice probabilities $c_H(F)$ and $c_G(F)$, which might seem the obvious extension of WARP.  Instead, WASRP limits how big the difference in probabilities can be, which must lie in the range $[-c_H(H\backslash G), +c_G(G\backslash H)]$.  Bandyopadhyay, Dasgupta, and Pattanaik (1999) and Dasgupta and Pattanaik (2007) provide the following argument for the upper bound.  When the choice set changes from $G$ to $H$, this rules out ``competing" choices in $G\backslash H$ to the alternatives in $F$.  This effect may raise the probability $c_H(F)$ that the chosen alternative lies in $F$.  But this increase, viz.~the difference $c_H(F) - c_G(F)$, should be bounded above by the original choice probability for $G\backslash H$, viz.~$c_G(G\backslash H)$.  The argument for the lower bound is analogous.

Call a subfamily $\cal G$ of $\cal F$ a nonempty $\pi$-system on $\Omega$ if the family $\cal G \cup \{\emptyset\}$ is a $\pi$-system.  Equivalently, if $G, H \in \cal G$ with $G \cap H \not= \emptyset$, then $G \cap H \in \cal G$.

\begin{theorem} \label{thm-level-1}
Fix a measurable space $(\Omega, \cal F)$ and a nonempty $\pi$-system $\cal G$ of choice sets.  The family of PCPSs $G \mapsto p_G(\cdot)$ coincides with the family of point stochastic choice rules $G \mapsto c_G(\cdot)$ satisfying WASRP.
\end{theorem}

\begin{proof}
Fix a PCPS $p$ relative to $\cal G$.  We need to show that $p$, as a stochastic choice rule, satisfies WASRP.  It is immediate that $p$ is a point stochastic choice rule, so we show that WASRP is satisfied.  Fix $G, H \in \cal G$ and $F \in \cal F$ with $F \subseteq G \cap H$.  Case (i): $s(G) \in G \backslash H$.  Then $p_G(G \backslash H) = 1$, from which Inequality~\ref{eq8} is satisfied.  Case (ii): $s(G) \in F$.  Then $p_G(F) = 1$, from which Inequality~\ref{eq8} is again satisfied.  Case (iii): $s(G) \in (G \cap H) \backslash F$.  Then $p_G(F) = 0$ and $p_G(G \backslash H) = 0$, so, in light of Inequality~\ref{eq8}, we need to show that $p_H(F) = 0$.  Suppose not, so that $s(H) \in F$ and $p_H(F) = 1$.  Then $p_H(G \cap H) = 1$.  Since $\cal G$ is a nonempty $\pi$-system, we have $G \cap H \in \cal G$.  The chain rule then gives $p_H(F) = p_H(G \cap H) \times p_{G \cap H}(F)$, from which we find $p_{G \cap H}(F) = 1$.  Using $p_G(F) = p_G(G \cap H) \times p_{G \cap H}(F)$, we find $p_G(F) = 0$ (as already noted), and $p_G(G \cap H) = 1$ (since $s(G) \in G \cap H$), from which $p_{G \cap H}(F) = 0$, a contradiction.  This establishes the forward direction of the proof.

For the reverse direction, fix a point stochastic choice rule $c$ relative to $\cal G$ satisfying WASRP.  We need to show that the chain rule holds: $c_G(E) = c_G(F) \times c_F(E)$ for every $E \subseteq F \subseteq G$ with $E \in \cal F$ and $F, G \in \cal G$.  Write WASRP in the form:
\begin{equation}~\label{eq9}
c_G(E) - c_F(E) \leq c_F(F\backslash G) \,\, \text{and} \,\, c_F(E) - c_G(E) \leq c_G(G\backslash F).
\end{equation}
Case (i): $s(G) \in E$ and $s(F) \in E$.  Then $c_G(E) = 1$, and therefore $c_G(F) = 1$, and also $c_F(E) = 1$, so that the chain rule is satisfied.  Case (ii): $s(G) \in E$ and $s(F) \notin E$.  Then $c_G(E) = 1$ and $c_F(E) = 0$.  But from the first case of Inequality~\ref{eq9}, using $c_F(F \backslash G) = c_F(\emptyset) = 0$, we then get $1 - 0 \leq 0$, a contradiction.  Case (iii): $s(G) \notin F$.  Then $c_G(E) = 0$ and $c_G(F) = 0$, so that the chain rule is again satisfied.   Case (iv): $s(G) \in F \backslash E$.  Then $c_G(E) = 0$ and $c_G(F) = 1$, and $c_G(G \backslash F) = 0$.  Now apply the second case of Inequality~\ref{eq9}.  Since $c_G(E) = 0$ and $c_G(G \backslash F) = 0$, we get $c_F(E) = 0$, and the chain rule holds once more, completing the proof.
\end{proof}

\section{Conditionally Consistent Mixtures of PCPSs} \label{sec:3}
To define Level 2 of our conditional probability hierarchy, we build a measurable space of PCPSs.  Fix a measurable space $(\Omega, \cal F)$ and a subfamily $\cal G$ of $\cal F$.  Let $\Pi$ be the set of all PCPSs $p$ relative to $\cal G$.  For $E \in \cal F$ and $G \in \cal G$, define the PCPS cylinder set:
\begin{equation}
[E, G]_\Pi = \{ p \in \Pi : p_G(E) = 1 \},
\end{equation}
and let ${\cal C}_\Pi$ be the collection of cylinders:
\begin{equation}
{\cal C}_\Pi =\bigl\{ [E, G]_\Pi : E \in {\cal F} \,\, \text{and} \,\, G \in \cal G \bigr\}.
\end{equation}

Since $p_G(G)=1$ for every $p\in\Pi$, we have $p_G(E)=p_G(E\cap G)$ for every $E\in\mathcal F$ and $G\in\mathcal G$. Consequently, $[E,G]_\Pi=[E\cap G,G]_\Pi$.  Therefore, without loss of generality, the event defining a cylinder may always be taken to be a measurable subset of its conditioning set.

The field of PCPSs, which we denote by ${\cal D}_\Pi$, is the smallest field of subsets of $\Pi$ containing ${\cal C}_\Pi$.

\begin{Def} \label{pi-realizable}
A stochastic choice rule $G \mapsto c_G(\cdot)$ is \textbf{realizable as a $\Pi$-based probability mixture} if there is a finitely additive probability measure $Q$ on $(\Pi, {\cal D}_\Pi)$ such that for all $E \in \cal F$ and $G \in \cal G$:
\begin{equation}~\label{eq75}
c_G(E) = Q([E, G]_\Pi).
\end{equation}
\end{Def}

We now define Level 2 of our hierarchy.

\begin{Def} \label{CC}
A probability measure $Q$ on $(\Pi, \mathcal D_\Pi)$ satisfies \textbf{Conditional Consistency} if for every $E \subseteq F \subseteq G$, with $E \in \cal F$ and $F, G \in \cal G$:
\begin{equation} \label{eq-CC-product}
Q([E, G]_\Pi) = Q([F, G]_\Pi) \times Q([E, F]_\Pi).
\end{equation}
\end{Def}

Dividing by $Q([F, G]_\Pi)$ when this term is strictly positive, we see that Conditional Consistency says that, conditioning on the event that the selected element of $G$ lies in $F$, the resulting probability measure on choices agrees with the probabilities when choosing directly from $F$.  We can think of Conditional Consistency as a chain rule applied at the level of the mixing measure.  Level 2 of our hierarchy consists of all stochastic choice rules that are realizable as $\Pi$-based probability mixtures for Conditionally Consistent $Q$.  The next result is immediate.

\begin{theorem} \label{thm:CC}
Fix a measurable space $(\Omega, \cal F)$ and a subfamily $\cal G$ of $\cal F$. Let $G \mapsto c_G(\cdot)$ be a stochastic choice rule realizable as a $\Pi$-based probability mixture $Q$.  Then $Q$ satisfies Conditional Consistency if and only if $c$ satisfies the chain rule.
\end{theorem}

\begin{proof}
First suppose that $Q$ satisfies Conditional Consistency.  Since $c_G$ is realized by $Q$, we have $c_G(E) = Q([E, G]_\Pi)$.  Now fix $E \subseteq F \subseteq G$ with $E \in \cal F$ and $F, G \in \cal G$.  Then Equation~\ref{eq-CC-product} holds, which becomes:
\begin{equation}
c_G(E) = c_G(F) \times c_F(E),
\end{equation}
so that $c$ satisfies the chain rule.  Conversely, suppose $c$ satisfies the chain rule.  Since:
\begin{equation}
Q([E, G]_\Pi) = c_G(E), \,\, Q([F, G]_\Pi) = c_G(F), \,\, Q([E, F]_\Pi) = c_F(E), 
\end{equation}
application of the chain rule for $c$ yields Equation~\ref{eq-CC-product}, and Conditional Consistency holds.
\end{proof}

\begin{corollary} \label{cor-CC}
Fix a measurable space $(\Omega, \cal F)$ and a subfamily $\cal G$ of $\cal F$.  Among $\Pi$-based probability mixtures, the Conditionally Consistent mixtures $Q$ are exactly those whose realized stochastic choice rules $G \mapsto c_G(\cdot)$ are CPSs.
\end{corollary}

The next question is the relationship between Level 2 of our hierarchy and stochastic choice rules that satisfy a (suitable generalization of) the usual Independence of Irrelevant Alternatives (IIA) axiom.

\begin{Def} \label{GIIA}
A stochastic choice rule $c: {\cal G} \to \Delta(\Omega)$ satisfies \textbf{Generalized Independence of Irrelevant Alternatives} (\textbf{GIIA}) if for every $G, H \in \cal G$, and every $E, F \in \cal F$ with $E, F \subseteq G \cap H$, we have:
\begin{equation} \label{eq10}
c_G(E) \times c_H(F) = c_H(E) \times c_G(F).
\end{equation}
\end{Def}

It is clear that GIIA implies the usual statement of IIA in terms of probability ratios, in the case that $c_H(E) \times c_H(F) \not= 0$.  Luce (1959) already noted the equivalence between CPSs and stochastic choice rules satisfying IIA, for the case of finite choice sets and strictly positive probabilities.  Cerreia-Vioglio et al.~(2021) extend this equivalence to the finitary case and nonnegative probabilities.  The next result is a further extension.  The forward direction also extends the proof of Theorem 6 in R\'enyi (1955). 

\begin{theorem} \label{thm-level-2}
Fix a measurable space $(\Omega, \cal F)$ and a nonempty $\pi$-system $\cal G$ of choice sets.  The family of CPSs $G \mapsto p_G(\cdot)$ coincides with the family of stochastic choice rules $G \mapsto c_G(\cdot)$ satisfying Generalized Independence of Irrelevant Alternatives (GIIA).
\end{theorem}

\begin{proof}
Fix a CPS $G \mapsto p_G(\cdot)$.  Also fix $G, H \in \cal G$ and $E, F \in \cal F$, where $E, F \subseteq G \cap H$.  Suppose first that $G \cap H \not= \emptyset$.  Since $\cal G$ is a nonempty $\pi$-system, we have $G \cap H \in \cal G$.  Using the equivalent R\'enyi (1955) form of Condition 2 of a CPS, we can write:
\begin{align}
p_G(E) & = p_G(E \cap H) = p_G(H) \times p_{G \cap H}(E), ~\label{eq11}\\
p_G(F) & = p_G(F \cap H) = p_G(H) \times p_{G \cap H}(F), ~\label{eq12}\\
p_H(E) & = p_H(E \cap G) = p_H(G) \times p_{G \cap H}(E), ~\label{eq13}\\
p_H(F) & = p_H(F \cap G) = p_H(G) \times p_{G \cap H}(F). ~\label{eq14}
\end{align}
Multiplying the left sides of Equations~\ref{eq11} and ~\ref{eq14}, and the left sides of Equations~\ref{eq12} and~\ref{eq13}, we obtain:
\begin{equation}~\label{15}
p_G(E) \times p_H(F) = p_H(E) \times p_G(F),
\end{equation}
as required.  If $G \cap H = \emptyset$, then $E = F = \emptyset$, and both sides of Equation~\ref{15} are $0$.  This establishes the forward direction of the proof.

For the reverse direction, suppose that Equation~\ref{15} holds and $H \subseteq G$, and set $F = H$.  We obtain
\begin{equation}~\label{eq16}
c_G(E) \times c_H(H) = c_H(E) \times c_G(H),
\end{equation}
from which, since $c_H(H) = 1$, our Condition 2 of a CPS is satisfied.
\end{proof}

When $\cal G$ is a nonempty $\pi$-system, Corollary \ref{cor-CC} and Theorem \ref{thm-level-2} together imply that Level-2 stochastic choice rules satisfy GIIA.  For remarks relevant to a full characterization, see Section \ref{sec:7}.d.

\section{Probabilistic Mixtures of PCPSs} \label{sec:4}
For Level 3 of our hierarchy, we drop the assumption of Conditional Consistency and allow arbitrary probabilistic mixtures of PCPSs.  Formally, Level 3 consists of all stochastic choice rules that arise as $\Pi$-based probability mixtures per Definition~\ref{pi-realizable}.  The characterization we provide is a de Finetti-style coherence condition.

We will use the formalism of positive partial charges (Rao and Rao, 1983, Definition 3.2.2).

\begin{Def}\label{charge}
Fix an arbitrary set $\X$ and let $\cal E$ be a family of subsets of $X$.  A function $\mu : {\cal E} \rightarrow \mathbb{R}^+$ is a \textbf{positive real partial charge on} $(X, {\cal E})$ if:
\begin{equation}~\label{eq64}
\sum_{i = 1}^m {\rm {\bf 1}}_{C_i} \leq \sum_{j = 1}^n {\rm {\bf 1}}_{D_j}
\end{equation}
pointwise on $X$, for any $C_i, D_j \in \cal E$, implies:
\begin{equation}~\label{eq65}
\sum_{i = 1}^m \mu(C_i) \leq \sum_{j = 1}^n \mu(D_j).
\end{equation}
\end{Def}

\begin{Def} \label{PCPS-coherence}
A stochastic choice rule $G \mapsto c_G(\cdot)$ satisfies \textbf{PCPS-coherence} if the assignment given by
\begin{equation} \label{eq-assignment}
\mu([E, G]_\Pi) = c_G(E),
\end{equation}
is a well-defined positive real partial charge on $(\Pi, {\cal C}_\Pi)$.
\end{Def}

This definition says that: (i) if two cylinder sets satisfy $[E, G]_\Pi = [F, H]_\Pi$, then $c_G(E) = c_{H}(F)$, so that the partial charge is well-defined; and (ii) Inequality~\ref{eq64} for cylinder sets $C_i, D_j$ implies Inequality~\ref{eq65} when $\mu$ is defined by Equation~\ref{eq-assignment}.  The definition has a natural interpretation as a de Finetti-like no Dutch book condition (de Finetti, 1937, 1974).  The indicator ${\rm {\bf 1}}_{[E, G]_\Pi}$ of a cylinder set $[E, G]_\Pi$ can be viewed as a bet on $\Pi$, which pays $1$ at a PCPS $p$ if $p_G(E) = 1$ and pays $0$ otherwise.  The stochastic choice probability $c_G(E)$ is interpreted as the price (or pr\'evision in the language of de Finetti, 1937) assigned to this bet.  PCPS-coherence requires that if one finite portfolio of PCPS-cylinder bets pays no more than another finite portfolio at every PCPS in $\Pi$, then the price assigned by $c$ to the first portfolio must be less than or equal to the price assigned by $c$ to the second portfolio.  In short, PCPS-coherence rules out a finite Dutch book on PCPS-cylinder bets.  Clark (1996) appears to be the pioneering paper on the connection between stochastic choice and the de Finetti no Dutch book requirement, where a characterization of random utility on infinite sets is given via this condition.  Our approach is an analog for PCPS-based stochastic choice.

An important validation of this approach is that when $\cal G$ is the finitary family of choice sets, our condition of PCPS-coherence reduces to the Block-Marschak condition (Block and Marschak, 1960) -- to be precise, to the finitary generalization of this condition due to Cohen (1980, Definition 2.4).  (Recall the finitary family of choice sets consists of all nonempty finite subsets of $\Omega$.)  We state this formally later in this section.

We now prove our characterization of Level 3 of our hierarchy.  We will use the extension result for charges (Rao and Rao, 1983, Theorem 3.2.10).

\begin{proposition} \label{extension}
Fix an arbitrary set $X$, let $\cal E$ be a family of subsets of $X$ with $X \in \cal E$, and suppose $\mu$ is a positive real partial charge on $(X, \cal E)$ with $\mu(X) = 1$.  Then, for any field $\cal D$ containing $\cal E$, there is a finitely additive probability measure $\widetilde \mu$ on $(X, \cal D)$ that extends $\mu$.
\end{proposition}

We apply Proposition~\ref{extension} with $X = \Pi$, ${\cal E} = {\cal C}_\Pi$, and ${\cal D} = {\cal D}_\Pi$.

\begin{theorem} \label{thm-level-3}
Fix a measurable space $(\Omega, \cal F)$ and a subfamily $\cal G$ of $\cal F$.  A stochastic choice rule ${\cal G} \mapsto c_G(\cdot)$ is realizable as a $\Pi$-based probability mixture if and only if it satisfies PCPS-coherence.
\end{theorem}

\begin{proof}
Suppose first that $c$ is realizable as a $\Pi$-based probability mixture.  That is, there is a finitely additive probability measure $Q$ on $(\Pi, {\cal D}_\Pi)$ with $c_G(E) = Q([E, G]_\Pi)$ for every $E \in \cal F$ and $G \in \cal G$.  Define $\mu$ according to Equation~\ref{eq-assignment}.  This is well-defined since if $[E, G]_\Pi = [F, H]_\Pi$, we have:
\begin{equation}
c_G(E) = Q([E, G]_\Pi) = Q([F, H]_\Pi) = c_H(F),
\end{equation}
and so $\mu=Q$ on ${\cal C}_\Pi$.  Suppose Inequality \ref{eq64} holds pointwise for $C_i, D_j \in {\cal C}_\Pi$.  Since $Q$ is a nonnegative finitely additive measure, the pointwise inquality implies:
\begin{equation}
\sum_{i=1}^m Q(C_i) \leq \sum_{j=1}^n Q(D_j).
\end{equation}
Since $\mu = Q$ on ${\cal C}_\Pi$, this is Inequality~\ref{eq65}.  This establishes that $\mu$ is a positive real partial charge on $(\Pi, {\cal C}_\Pi)$, so that $c$ satisfies PCPS-coherence.

Conversely, suppose $c$ satisfies PCPS-coherence, that is, the assignment of Equation~\ref{eq-assignment} is a well-defined positive real partial charge on $(\Pi, {\cal C}_\Pi)$.  For any $G \in \cal G$, we have $[G, G]_\Pi = \Pi$ snce $p_G(G) = 1$ for every $p \in \Pi$, and so $\Pi \in {\cal C}_\Pi$.  Moreover, since $c$ is a stochastic choice rule, $\mu(\Pi) = \mu([G, G]_\Pi) = c_G(G) = 1$.  By the extension theorem for positive real partial charges (Proposition~\ref{extension}), we can extend $\mu$ (not necessarily uniquely) to a finitely additive probability measure $Q$ on the field ${\cal D}_\Pi$ generated by ${\cal C}_\Pi$.  This says, in particular, that for every $E \in \cal F$ and $G \in \cal G$:
\begin{equation}
Q([E, G]_\Pi) = \mu([E, G]_\Pi) = c_G(E),
\end{equation}
so that $c$ is realizable as a $\Pi$-based probability mixture.
\end{proof}

We validate our approach to defining Level 3 of our hierarchy by establishing that PCPS-coherence reduces to the generalized Block-Marschak condition (Cohen, 1980, Definition 2.4) in the finitary case.

\begin{Def} \label{block-marschak}
Fix an arbitrary set $\Omega$, let $\cal G$ be the finitary family of subsets of $\Omega$, and let $\cal F$ be the $\sigma$-field generated by $\cal G$.  A stochastic choice rule $G \mapsto c_G(\cdot)$ satisfies the \textbf{Generalized Block-Marschak condition} if for every $H \in \cal G$, $\alpha \in H$, and $G \subseteq H\backslash\{\alpha\}$, the coefficient:
\begin{equation} \label{eq-bm}
q(\alpha, G, H) = \sum_{E \subseteq G} (-1)^{|G\backslash E|} \, c_{H\backslash E}(\{\alpha\}),
\end{equation}
is nonnegative. 
\end{Def}

The proof of the next result is in Appendix \hyperref[app-a]{A}.

\begin{proposition} \label{prop-bm}
Fix an arbitrary set $\Omega$, let $\cal G$ be the finitary family of subsets of $\Omega$, and let $\cal F$ be the $\sigma-$field generated by $\cal G$.  A stochastic choice rule $G \mapsto c_G(\cdot)$ satisfies PCPS-coherence if and only if it satisfies the Generalized Block-Marschak condition.
\end{proposition}

The characterization of Level 3 of our conditional probability hierarchy via PCPS-coherence is related in spirit to the recent Dutch book theorem of Catonini and Lanzani (2026).  Their Dutch book condition characterizes when a system of conditional beliefs can be represented by a global object -- namely, a (prior) lexicographic conditional probability system (Blume, Brandenburger, and Dekel, 1991).  We likewise ask when a family of locally specified choice probabilities (a stochastic choice rule) can be represented by a global object -- more precisely, by a probabilistic mixture of global objects, namely, PCPSs.  The analysis of Catonini and Lanzani (2026) appears relevant to addressing the open representation question we pose in Section~\hyperref[sec:7d]{7.d}.

\section{Signed Probabilistic Mixtures of PCPSs} \label{sec:5}
Level 4 of our conditional probability hierarchy nests Level 3 by relaxing the requirement that the mixing $Q$ is nonnegative.  To define this level, we employ the same measurable structure as for probabilistic mixtures and replace unsigned with signed measures.

\begin{Def} \label{signed-realizable}
A stochastic choice rule $G \mapsto c_G(\cdot)$ is \textbf{realizable as a $\Pi$-based signed probability mixture} if there is a finitely additive signed probability measure $Q$ on $(\Pi, {\cal D}_\Pi)$ such that for all $E \in \cal F$ and $G \in \cal G$:
\begin{equation} \label{eq95}
c_G(E) = Q([E, G]_\Pi).
\end{equation}
\end{Def}

The introduction of negativity in this definition can be seen as a formal move, or it can be given a behavioral interpretation (Dogan and Yildi, 2023; Brandenburger et al., 2026).  (See Example~\ref{ex-b3} in Appendix \hyperref[app-b]{B}.)  We partially characterize Level 4 by transferring a result from Dogan and Yildiz (2023).  Fix finite $\Omega$ and let ${\cal F} = 2^\Omega$ and ${\cal G} = 2^\Omega\backslash\{\emptyset\}$.  Dogan and Yildiz
(2023, Theorem 2) show that every stochastic choice rule on this domain admits a signed probabilistic mixture over total orders.  Under these domain assumptions, there is a bijection between total orders and PCPSs.  (This follows from Proposition~\ref{PCPS-TO-bijection} in the next section.)  Therefore, a signed-mixture representation over total orders transfers directly to a signed-mixture representation over PCPSs, and we get an analogous result.

\begin{theorem} \label{thm-level-4}
Fix a finite set $\Omega$ and let ${\cal F} = 2^\Omega$ and ${\cal G} = 2^\Omega\backslash{\{\emptyset\}}$.  Every stochastic choice rule $G \mapsto c_G(\cdot)$ is realizable as a $\Pi$-based signed probability mixture. 
\end{theorem}

Saito (2018) proves a more general result on realization of stochastic choice rules as signed mixtures of orders, where the family $\cal G$ of choice sets can be an arbitrary subset of $2^\Omega\backslash\{\emptyset\}$.  Saito (2018, Corollary 6(ii) and Footnote 22) demonstrates how Theorem 2 in Dogan and Yildiz (2023) follows.  The proof of the Dogan-Yildiz result works via a (significant) extension of the Ford-Fulkerson Theorem from combinatorial matrix theory (Ford and Fulkerson, 2015) to allow for negative row and column sums.  The generalization of this method to an infinite choice domain $\Omega$ is, we believe, open -- and the question of what is the analog to Theorem~\ref{thm-level-4} for infinite $\Omega$ is open.

\section{PCPSs and Total Orders} \label{sec:6}
We have constructed our hierarchy using PCPSs as the basic building block.  We now examine the relationship between our approach and the existing one of building stochastic choice rules from total orders.  Fix a measurable space $(\Omega, \cal F)$ and let $\cal G$ be a subfamily of $\cal F$ consisting of (certain) nonempty finite sets in $\cal F$.  The assumption that $\cal G$ consists of finite sets is obviously more restrictive than our standing assumptions on $\cal G$ so far.  It allows us to avoid issues of well-ordering.  (See also Section~\ref{sec:7}.d.)

There is a natural map from total orders on $\Omega$ to selections $s : {\cal G} \rightarrow \Omega$.  Formally, given a total order $\rhd$ on $\Omega$, define the selection $s^\rhd$ by:
\begin{equation} \label{eq-well-defined}
s^\rhd(G) = x \in G \,\, \text{such that} \,\, x \rhd y \,\, \text{for all} \,\, y \in G\backslash \{x\}.
\end{equation}
Finiteness of each $G$ ensures this operation is well-defined.  We write $p^\rhd$ for the putative PCPS defined by the selection $s^\rhd$.  The next proposition says that $p^\rhd$ is indeed a PCPS.  Let $\Lambda$ denote the set of all total orders on $\Omega$ and recall our earlier notation $\Pi$ for the set of all PCPSs relative to some given family $\cal G$.

\begin{proposition} \label{pro-well-defined}
Fix a measurable space $(\Omega, \cal F)$ and a subfamily $\cal G$ of $\cal F$ consisting of (certain) nonempty finite sets.  The operation $f$ defined by:
\begin{equation}
f(\rhd) = p^\rhd,
\end{equation}
defines a map from $\Lambda$ to $\Pi$.
\end{proposition}

\begin{proof}
Fix $E \subseteq F \subseteq G$, with $E \in \cal F$ and $F, G \in \cal G$.  We need to show that the chain rule holds, that is:
\begin{equation} \label{eq-chain-again}
p_G^\rhd(E) = p^\rhd_G(F) \times p^\rhd_F(E).
\end{equation}
If $s^\rhd(G) \notin F$, then $p^\rhd_G(F) = 0$.  Since $E \subseteq F$, we also have $p^\rhd_G(E) = 0$, and Equation~\ref{eq-chain-again} holds.  If $s^\rhd(G) \in F$, then $s^\rhd(G) = s^\rhd(F)$, so that: (i) $p^\rhd_G(F) = 1$; and (ii) $p^\rhd_G(E) = p^\rhd_F(E)$.  Equation~\ref{eq-chain-again} again holds.
\end{proof}

We next present two examples that illustrate the behavior of the map $f : \Lambda \rightarrow \Pi$.

\begin{Ex} \label{ex-not-exist}
Let $\Omega = \{x, y, z\}$, ${\cal F} = 2^\Omega$, and ${\cal G} = \{\{x, y\}, \{y, z\}, \{z, x\}, \{x\}, \{y\}, \{z\}\}$.  Note that $\cal G$ is a nonempty $\pi$-system.  Consider the PCPS given by:
\begin{align}
p_{\{x, y\}}(\{x\}) &= p_{\{y, z\}}(\{y\}) = p_{\{z, x\}}(\{z\}) = 1, \\
p_{\{x\}}(\{x\}) &= p_{\{y\}}(\{y\}) = p_{\{z\}}(\{z\}) = 1.
\end{align}
A total order $\rhd$ that maps under $f$ to the PCPS $p$ would need $x \rhd y \rhd z \rhd x$, a contradiction.  We conclude that the map $f$ need not be surjective.
\end{Ex}

\begin{Ex} \label{ex-not-injective}
Let $\Omega = \{x, y, z\}$, ${\cal F} = 2^\Omega$, and ${\cal G} = \{\{x, y\},\{y, z\}, \{y\}, \{x, y, z\}\}$.  Note that $\cal G$ is a nonempty $\pi$-system and is closed under (finite) unions.  We define a PCPS, for $E \subseteq \Omega$, by: 
\begin{equation}
p_G(E) = \delta_y(E),
\end{equation}
where $\delta_y$ is the Dirac measure on $y$.  Now consider the two total orders on $\Omega$:
\begin{equation}
y \rhd_1 z \rhd_1 x \,\,\, \text{and} \,\,\, y \rhd_2 x \rhd_2 z.
\end{equation}
It is easily verified that both $\rhd_1$ and $\rhd_2$ map to the PCPS $p$ under $f$.  We conclude that   the map $f$ need not be injective.
\end{Ex}

We next establish conditions under which the map $f$ is, respectively, surjective and bijective.  We will use the fact that every PCPS admits a selection satisfying the Chernoff property (Chernoff, 1954, Postulate 4).  (If $\cal G$ is a nonempty $\pi$-system, the Chernoff property is equivalent to the WASRP axiom from Section~\ref{sec:2}.  But we do not want to assume this closure condition yet in this section.)

\begin{proposition}\label{prop-chernoff}
Fix a PCPS $p: G \mapsto p_G(\cdot)$.  There is a selection $s: {\cal G} \to \Omega$ generating $p$ that satisfies the Chernoff property:
\begin{equation}
F, G \in {\cal G} \,\, \text{with} \,\, F \subseteq G \,\, \text{and} \,\, s(G) \in F \Rightarrow s(F) = s(G).
\end{equation}
\end{proposition}

\begin{proof}
Define an equivalence relation $\sim$ on $\mathcal G$ by $F \sim G$ if $p_F(\cdot) = p_G(\cdot)$.  For each equivalence class $C \in \mathcal G/{\sim}$, let $\mu_C$ denote the common probability measure associated with sets in $C$. Then $\mu_C$ is a Dirac measure since $p$ is a PCPS.  Choose $x_C \in \Omega$ such that $\mu_C = \delta_{x_C}$.  For every $G \in C$, we have:
\begin{equation}
\delta_{x_C}(G) = \mu_C(G) = p_G(G) = 1,
\end{equation}
and so $x_C \in G$.  Therefore, writing $[G]$ for the equivalence
class of $G$, the assignment $s(G) = x_{[G]}$ defines a selection on $\mathcal G$.  Moreover:
\begin{equation}
p_G = \mu_{[G]} = \delta_{x_{[G]}} = \delta_{s(G)},
\end{equation}
for every $G \in\mathcal G$, so that $s$ generates $p$.

Fix $F, G \in \mathcal G$ with $F\subseteq G$ and $s(G )\in F$.
Because $p_G = \delta_{s(G)}$, we have $p_G(F) = 1$.  From this, for every $E \in \mathcal F$:
\begin{align}
p_G(E) &= p_G(E \cap F) = p_G(F) \times p_F(E\cap F)\\
&= p_F(E \cap F) = p_F(E),
\end{align}
where the first equality follows from $p_G(F) = 1$, the second from the chain rule, and the last from $p_F(F) = 1$.  Therefore, $p_F = p_G$ and $F \sim G$.  By the construction of $s$, it follows that $s(F) = s(G)$, as required.
\end{proof}

We now state and prove the surjection result.

\begin{proposition} \label{PCPS-TO-Surjection}
Fix a measurable space $(\Omega, \cal F)$ and a subfamily $\cal G$ of $\cal F$ that consists of (certain) nonempty finite sets and is closed under finite unions.  Then the map $f$ is a surjection.
\end{proposition}

\begin{proof}
Fix a PCPS $p$.  By Proposition~\ref{prop-chernoff}, we can choose a selection $s$ that generates $p$ and satisfies the Chernoff property. Define a binary relation $\rhd_p$ on $\Omega$ by:
\begin{equation}
x \rhd_p y \,\, \text{if} \,\, \text{there is a} \,\, G \in {\cal G} \,\, \text{with} \,\, x = s(G) \,\, \text{and} \,\, y \in G \backslash \{x\}.
\end{equation}
We first show that $\rhd_p$ is acyclic.  Suppose not.  Take a cycle of minimal length and write:
\begin{equation}
x_1 \rhd_p x_2 \rhd_p \cdots \rhd_p x_n \rhd_p x_1.
\end{equation}
By minimality, each $x_1, \ldots, x_n$ are distinct.  For each $i = 1, 2, \ldots, n$, choose $G_i \in \cal G$ such that $s(G_i) = x_i$ and $x_{i+1} \in G_i$, where indices are interpreted cyclically, so $x_{n+1} = x_1$, $x_0=x_n$, and $G_0=G_n$.  Since $\cal G$ is closed under finite unions:
\begin{equation}
U = \bigcup_{i=1}^n G_i \in \cal G.
\end{equation}
Let $z = s(U)$.  Since $z \in U$, there is some $j$ with $z \in G_j$.  Because $G_j \subseteq U$ and $s(U) \in G_j$, Proposition~\ref{prop-chernoff} gives $s(G_j) = z$.  Since $s(G_j) = x_j$, we have $z = x_j$.  But $x_j \in G_{j-1}$ by construction, from which $s(U) = x_j \in G_{j-1} \subseteq U$.  Another application of Proposition~\ref{prop-chernoff} gives $s(G_{j-1}) = s(U) = x_j$.  However, we also have $s(G_{j-1}) = x_{j-1}$ by construction.  Therefore $x_{j-1} = x_j$, contradicting the definition of the cycle.  This establishes that $\rhd_p$ is acyclic.

The transitive closure of $\rhd_p$ is therefore a strict partial order.  By the Szpilrajn order-extension theorem (Szpilrajn, 1930), it extends to a (strict) total order $\rhd$ on $\Omega$.  Note that for every $G \in \cal G$ and $y \in G \backslash \{s(G)\}$, we have $s(G) \rhd_p y$ and hence $s(G) \rhd y$.  It follows that $s(G)$ is the $\rhd$-maximal element of $G$.  We conclude that the PCPS induced by $\rhd$ coincides with $p$, and so the map $f$ is a surjection.
\end{proof}

In the next result, which yields a bijection, we say that $\mathcal F$ separates points if for each $x, y \in \Omega$ with $x \not= y$, there is a set $E \in \mathcal F$ with $x \in E$ and $y \notin E$.

\begin{proposition} \label{PCPS-TO-bijection}
Fix a measurable space $(\Omega, \cal F)$, where $\cal F$ separates points, and a subfamily $\cal G$ of $\cal F$ which consists of (certain) nonempty finite sets, is closed under finite unions, and includes all two-element choice sets.  Then the map $f$ is a bijection.
\end{proposition}

\begin{proof}
Proposition~\ref{PCPS-TO-Surjection} implies that the map $f$ is surjective. It remains to prove injectivity.  Suppose that two (strict) total orders $\rhd_1$ and $\rhd_2$ induce the same PCPS ${\cal G} \mapsto p_G(\cdot)$.  If these orders are distinct, there are distinct $x, y\in\Omega$ with $x\rhd_1 y$ and $y\rhd_2 x$.  Since $\mathcal F$ separates points, there is a set $E \in \mathcal F$ with $x \in E$ and $y \notin E$. Since $\{x, y\} \in \mathcal G$, the order $\rhd_1$ selects $x$ from this set, while $\rhd_2$ selects $y$.  Therefore:
\begin{equation}
p^{\rhd_1}_{\{x, y\}}(E) = 1 \,\, \text{and} \,\, p^{\rhd_2}_{\{x, y\}}(E) = 0,
\end{equation}
which contradicts the assumption that the two orders induce the same PCPS.  It follows that $\rhd_1 = \rhd_2$, so that $f$ is injective, as required.
\end{proof}

Proposition~\ref{PCPS-TO-Surjection} is tight.  Example~\ref{ex-not-exist} shows that we cannot drop closure under unions (even if we add that $\cal G$ is a nonempty $\pi$-system).  Proposition~\ref{PCPS-TO-bijection} is also tight.  Example~\ref{ex-not-injective} shows that we cannot drop the requirement that $\cal G$ contain all two-element choice sets.

The results in this section justify the discussion in the Introduction about the parsimony of the PCPS approach.  Under our standing assumptions, the ``best case" is Proposition~\ref{PCPS-TO-Surjection}, which establishes a surjection.  We say ``best case" because we do not actually require finiteness of choice sets or closure under finite unions.  But, our point is that even in this case, more than one total order may map to the same PCPS.  Example~\ref{ex-not-injective} is a demonstration.  The two total orders $\rhd_1$ and $\rhd_2$ there are indistinguishable because the choice set $\{z, x\}$ that would separate them is not present in $\cal G$.  Indeed, if the family of choice sets $\cal G$ is specified first and behavior examined second, then one can argue that $\rhd_1$ and $\rhd_2$ should be identified, as happens under the PCPS in the example.  Conversely, if preference is thought of as prior to the specification of choice sets, then total orders are a natural starting point.

\section{Discussion}  \label{sec:7}
In this section, we discuss strict nesting of the levels of our hierarchy, cover some literature that refines our hierarchy, and mention an alternative hierarchy based on stochastic transitivity.
\vspace{0.1in}

\noindent\textbf{a. Strict Nesting of Levels}
It is immediate that Level 2 of our hierarchy strictly nests Level 1.  The well-known ``red bus-blue bus" Similarity Effect (Debreu, 1960; Tversky, 1972a, 1972b; McFadden, 1974) establishes that Level 3 strictly nests Level 2.  In Appendix \hyperref[app-b]{B}, we couch this effect in terms of an underlying set $\Omega = \{{\rm red\,\,bus}, {\rm blue\,\,bus}, {\rm taxi}\}$, choice sets ${\cal G} = \{\Omega, \{{\rm blue\,\,bus}, {\rm taxi}\}\}$, and probabilistic mixing over three PCPSs relative to $\cal G$.  This setup violates Conditional Consistency.

Level 4 strictly nests Level 3, as can be shown via the Attraction Effect (Huber, Payne, and Puto, 1982; Simonson, 1989) or Repulsion Effect (Aaker, 1991) in our framework.  The method is to use the fact that Level 3 satisfies the Regularity axiom.  (Expanding a choice set does not increase the probability assigned to any subset of the original choice set.)  The Attraction and Repulsion effects violate Regularity.  In Appendix \hyperref[app-b]{B}, we couch these effects in terms of signed probabilistic mixtures of PCPSs, which gives a direct demonstration that they lie in Level 4.
\vspace{0.1in}

\noindent\textbf{b. Non-Commutativity}
A natural starting conjecture would be that probabilistic mixing over PCPSs yields the family of CPSs.  Indeed, both PCPSs and CPSs satisfy a chain rule -- the first pointwise and the second probabilistically.  The Similarity Effect (Appendix \hyperref[app-b]{B}) shows that this conjecture is false, but it is worth articulating why this is the case at a conceptual level.

There are two operations involved.  One is enforcing the chain rule.  The other is probabilistic mixing, which takes deterministic choice rules to stochastic choice rules.  Starting from a point stochastic choice rule, enforcing the chain rule produces a PCPS, and probabilistic mixing of PCPSs yields a Level-3 object.  Alternatively, we can first probabilistically mix point stochastic choice rules to obtain a general stochastic choice rule, and then enforce the chain rule, thereby obtaining a CPS.  In general, this is a different answer.  Evidently, the two operations of enforcement of the chain rule and probabilistic mixing do not commute.  We think this non-commutativity is a useful complementary way of understanding how our hierarchy works.
\vspace{0.01in}

\noindent\textbf{c. Related Literature} Our four-level hierarchy for stochastic choice arose organically from the starting point of the PCPS concept.  We do not offer a more fine-grained classification in this paper.  In this direction, one can distinguish between stochastic choice rules that do or do not satisfy Regularity.  A well-known example of a stochastic choice rule  that lies above the Block-Marschak conditions but below Regularity is the additive perturbed utility (APU) model due to Fudenberg, Iijima, and Strzalecki (2015).  This model family can encompass non-expected utility preferences arising from implementation costs and ambiguity aversion.  A theory that is the convex dual to APU is the variational preference model of Maccheroni, Marinacci, and Rustichini (2006), interpretable as a game against a ``malevolent" Nature.  Under restrictions on cost functions, the resulting stochastic choice models satisfy Regularity.

Lying above Regularity is the deliberately stochastic choice model of Cerreia-Vioglio et al.~(2019), where convex preferences over outcomes induce Regularity violations.  Information processing architectures can also lead to departures from Regularity, as in the leading analysis by Matejka and McKay (2015).  The divisive normalization model of Steverson, Brandenburger, and Glimcher (2019) employs a neuroscience-inspired information cost that likewise yields Regularity violations.  Finally, two-stage models involving choice of consideration set and then choice of alternative can readily accommodate non-Regular behavior (Kashaev and Aguiar, 2022).

Galichon (2022) shows how a nested logit model can be written as a mixture of Luce models, thereby allowing for the Compromise Effect and forms of the Attraction Effect that are consistent with Regularity, but not for an Attraction Effect that violates Regularity.  Behavioral models such as Gul, Natenzon, and Pesendorfer (2014) and Kovach and Tserenjigmid (2022) introduce additional latent structure -- attributes and focal sets, respectively -- to explain context effects.

A well-known hierarchy for stochastic choice is based on stochastic transitivity, which comes in three forms.  A stochastic choice rule $G \mapsto c_G(\cdot)$ satisfies strong stochastic transitivity (SST), medium stochastic transitivity (MST), or weak stochastic transitivity (WST), if $c_{\{x, y\}}(\{x\}) \geq 1/2$ and $c_{\{y, z\}}(\{y\}) \geq 1/2$ imply:
\begin{equation}~\label{eq96}
c_{\{z, x\}}(\{x\}) \geq
\begin{cases}
& \max \{c_{\{x, y\}}(\{x\}), c_{\{y, z\}}(\{y\})\}, \\
& \min \{c_{\{x, y\}}(\{x\}), c_{\{y, z\}}(\{y\})\}, \\
& \frac{1}{2},
\end{cases}
\end{equation}
respectively.  Rieskamp, Busemeyer, and Mellers (2006) propose an increasing ``rationality" hierarchy going from WST through IIA to SST.  He and Natenzon (2024) use a version of MST to characterize Fechnerian utility functions.  However, even Level 1 of our hierarchy does not satisfy WST.
\vspace{0.1in}

\noindent \textbf{Example 1 Contd.}  \textit{Repeating the previous setup, we find:}
\begin{equation}
p_{\{x, y\}}(\{x\}) = 1 \ge \dfrac{1}{2} \,\, \text{and} \,\, p_{\{y, z\}}(\{y\}) = 1 \ge \dfrac{1}{2}.
\end{equation}
\textit{WST then requires $p_{\{x, z\}}(\{x\}) \ge 1/2$, while, in fact, we have $p_{\{x, z\}}(\{x\}) = 0$.}
\vspace{0.1in}

The stochastic transitivity hierarchy appears to be a distinct organization of stochastic choice rules.
\vspace{0.15in}

\noindent\textbf{d. Open Directions} \label{sec:7d}
A significant question left open is whether Level 2 coincides with the family of stochastic choice rules that are CPSs.  Since every Level-2 stochastic choice rule is a CPS, the issue is whether every CPS is a probabilistic mixture of PCPSs -- equivalently, whether every CPS lies in Level 3.  If so, Theorem \ref{thm:CC} implies that any such mixture is Conditionally Consistent and therefore lies in Level 2.  We are especially interested in this question when $\cal G$ is a nonempty $\pi$-system, since then Theorem \ref{thm-level-2} would yield a full characterization of Level 2 in terms of the the GIIA axiom.

Another open direction, already mentioned in Section \ref{sec:5}, is to determine which stochastic choice rules on an infinite set $\Omega$ of elementary choice objects admit signed probabilistic mixture representations via PCPSs, and how such representations can be characterized in terms of choice axioms.  A third direction concerns the relationship between total orders and PCPSs -- which, in Section \ref{sec:6}, we considered when all choice sets in $\cal G$ are finite.  The examination of this relationship for more general families $\cal G$ would be of interest.

\renewcommand{\thesection}{A}
\setcounter{equation}{0}
\renewcommand\theequation{A.\arabic{equation}}
\setcounter{theorem}{0}
\renewcommand\thetheorem{A.\arabic{theorem}}
\setcounter{proposition}{0}
\renewcommand\theproposition{A.\arabic{proposition}}
\setcounter{Ex}{0}
\renewcommand\theEx{A.\arabic{Ex}}

\appendix

\section*{Appendix A: Proof of Proposition~\ref{prop-bm}} \label{app-a} 
We proceed in two stages.  First, we use the standard M\"obius-transform argument of stochastic choice theory (Falmagne, 1978; Barber\'a and Pattanaik, 1986; Rota, 1964) to show that the PCPS-cylinder coefficients defined in Equation~\ref{PCPS-BM} are pointwise nonnegative.  We then invoke PCPS-coherence to transfer this pointwise  to the corresponding cylinder prices, in order to obtain the Generalized Block-Marschak condition.  For the converse direction, we restrict attention to a finite family of choice sets and use the classical finite Block-Marschak theorem to obtain PCPS-coherence.

Throughout, for $\alpha \in K\in \cal G$, we write:
\begin{equation}
[\alpha, K]_\Pi = [\{\alpha\}, K]_\Pi.
\end{equation}
First suppose that $G \mapsto c_G(\cdot)$ satisfies PCPS-coherence.  For $H \in \cal G$, $\alpha \in H$, and $G \subseteq H \backslash \{\alpha\}$, define:
\begin{equation}\label{PCPS-BM}
q_p(\alpha, G, H) = \sum_{E \subseteq G} (-1)^{|G\backslash E|} \, \mathbf{1}_{[\alpha, H\backslash E]_\Pi}(p).
\end{equation}

We show that $q_p(\alpha, G, H) \geq 0$ for every $p \in \Pi$.  For a PCPS $p \in \Pi$, write $s_p(K)$ for the point selected by $p$ from the finite choice set $K$.  Since $\cal G$ is the finitary family, Proposition~\ref{PCPS-TO-bijection} gives a unique strict total order $\rhd_p$ on $\Omega$ that induces $p$.  Equivalently, for distinct $x, y \in \Omega$, we have $x \rhd_p y$ if and only if $s_p(\{x, y\}) = x$.  In particular, $s_p(K)$ is the $\rhd_p$-maximal element of every $K \in \cal G$.

Define the set of choices in $H$ that block $\alpha$ under $p$ by:
\begin{equation}
B_p(\alpha, H) = \{\beta \in H\backslash\{\alpha\} : \beta \rhd_p \alpha\}.
\end{equation}
For any $E \subseteq G$, we have:
\begin{equation}
\mathbf{1}_{[\alpha, H\backslash E]_\Pi}(p)=1,
\end{equation}
if and only if $p$ selects $\alpha$ from $H \backslash E$.  Since $p$ selects the top-ranked element of each choice set, this is equivalent to saying that every choice in $H$ ranked above $\alpha$ has been removed.  Therefore:
\begin{equation}
\mathbf{1}_{[\alpha, H\backslash E]_\Pi}(p) = 1 \,\, \text{if and only if} \,\, B_p(\alpha, H) \subseteq E.
\end{equation}
It follows that:
\begin{equation}
q_p(\alpha, G, H) = \sum_{\substack{E \subseteq G \\ B_p(\alpha, H) \subseteq E}} (-1)^{|G\backslash E|}.
\end{equation}
If $B_p(\alpha, H)\nsubseteq G$, the sum is empty, and therefore $q_p(\alpha, G, H) = 0$.  If $B_p(\alpha,H) \subseteq G$, then, writing $B = B_p(\alpha,H)$:
\begin{equation}
q_p(\alpha, G, H) = \sum_{B\subseteq E\subseteq G} (-1)^{|G\backslash E|} = \sum_{L\subseteq G\backslash B} (-1)^{|(G\backslash B)\backslash L|} = (1 - 1)^{|G\backslash B|}.
\end{equation}

We have now shown:
\begin{equation}
q_p(\alpha, G, H) =
\begin{cases}
1, & \text{if} \,\, B_p(\alpha, H) = G, \\
0, & \text{otherwise}.
\end{cases}
\end{equation}
In particular, we have $q_p(\alpha, G, H) \ge 0$ for every $p \in \Pi$.  Equivalently, the following pointwise inequality holds on $\Pi$:
\begin{equation}
\sum_{\substack{E \subseteq G \\ |G\backslash E| \, \text{even}}}
\mathbf{1}_{[\alpha, H\backslash E]_\Pi} \ge \sum_{\substack{E \subseteq G \\ |G\backslash E| \, \text{odd}}} \mathbf{1}_{[\alpha, H\backslash E]_\Pi}.
\end{equation}
By PCPS-coherence, the same inequality holds after replacing each cylinder indicator by its price under $c$, that is:
\begin{equation}
\sum_{\substack{E \subseteq G \\ |G\backslash E| \, \text{even}}}
c_{H\backslash E}(\{\alpha\}) \ge \sum_{\substack{E\ \subseteq G \\ |G\backslash E| \, \text{odd}}} c_{H\backslash E}(\{\alpha\}),
\end{equation}
which says that $q(\alpha, G, H)\ge 0$.  This establishes that PCPS-coherence implies the Generalized Block-Marschak condition.

For the converse direction, suppose that  ${\cal G} \mapsto c_G(\cdot)$ satisfies the Generalized Block Marschak condition.  To establish PCPS-coherence, consider finite collections of cylinder sets:
\begin{align}
&[E_1, G_1]_\Pi, \ldots, [E_m, G_m]_\Pi, \\
&[F_1, H_1]_\Pi, \ldots, [F_n,H_n]_\Pi,
\end{align}
with:
\begin{equation}
\sum_{i=1}^m \mathbf{1}_{[E_i, G_i]_\Pi}(p) \le \sum_{j=1}^n \textbf{1}_{[F_j, H_j]_\Pi}(p),
\end{equation}
for every $p \in \Pi$. 
By the observation following the definition of PCPS cylinders, we may replace each $E_i$ by $E_i\cap G_i$ and each $F_j$ by $F_j\cap H_j$, so that, without loss of generality, $E_i \subseteq G_i$ and $F_j \subseteq H_j$.  Let:
\begin{equation}
S= (\bigcup_{i=1}^m G_i) \cup (\bigcup_{j=1}^n H_j).
\end{equation}
Since only finitely many choice sets are involved, the set $S$ is finite.  Let $c^S$ denote the restriction of $c$ to the nonempty subsets of $S$.  The Generalized Block-Marschak condition for $c$ reduces to the ordinary Block-Marschak condition on this finite domain.  By the classical Block-Marschak theorem (Falmagne, 1978), there is a probability measure $Q^S$ on the set $\Lambda^S$ of strict total orders on $S$ such that for every $L \subseteq K \subseteq S$ with $K \not= \emptyset$:
\begin{equation} \label{eq-Q-s}
c_K(L) = Q^S\bigl(\{\rhd \in \Lambda^S : \max_\rhd K \in L\}\bigr).
\end{equation}

Fix $\rhd \in \Lambda^S$.  By the Szpilrajn order-extension theorem (Szpilrajn, 1930), we can extend $\rhd$ to a strict total order $\rhd^*$ on $\Omega$.  Let $p^{\rhd^*}$ be the PCPS induced by $\rhd^*$.  By assumption:
\begin{equation}
\sum_{i=1}^m \mathbf{1}_{[E_i, G_i]_\Pi}(p^{\rhd^*}) \le \sum_{j=1}^n \textbf{1}_{[F_j, H_j]_\Pi}(p^{\rhd^*}).
\end{equation}
Since all sets $G_i, H_j$ are contained in $S$, the values of these indicators depend only on the restriction of $\rhd^*$ to S, which is $\rhd$.  Therefore, for every $\rhd \in \Lambda^S$:
\begin{equation} \label{eq-new-max}
\sum_{i=1}^m \mathbf{1}_{\{\max_\rhd G_i \in E_i\}} \le \sum_{j=1}^n \textbf{1}_{\{\max_\rhd H_j \in F_j\}}.
\end{equation}
Integrating Equation~\ref{eq-new-max} under $Q^S$:
\begin{align}
     \sum_{i=1}^m Q^S\bigl(\{\rhd \in \Lambda^S : \max_\rhd G_i \in E_i\}\bigr) 
\le \sum_{j=1}^n Q^S\bigl(\{\rhd \in \Lambda^S : \max_\rhd H_j \in F_j\}\bigr). 
\end{align}
Using Equation~\ref{eq-Q-s} yields:
\begin{align}
\sum_{i=1}^m c_{G_i}(E_i) \le \sum_{j=1}^n c_{H_j}(F_j).
\end{align}
This establishes that $c$ satisfies PCPS-coherence, completing the converse direction.  Proposition~\ref{prop-bm} is now proved.

\renewcommand{\thesection}{B}
\setcounter{equation}{0}
\renewcommand\theequation{B.\arabic{equation}}
\setcounter{theorem}{0}
\renewcommand\thetheorem{B.\arabic{theorem}}
\setcounter{proposition}{0}
\renewcommand\theproposition{B.\arabic{proposition}}
\setcounter{Ex}{0}
\renewcommand\theEx{B.\arabic{Ex}}

\section*{Appendix B: Strict Nesting of Levels} \label{app-b} 
The Similarity Effect (Debreu, 1960; Tversky, 1972a, 1972b; McFadden, 1974) establishes that Level 3 strictly nests Level 2 of our conditional choice hierarchy.  Here, we couch this effect in terms of probabilistic mixtures of PCPSs.

\begin{Ex}
Let $\Omega = \{{\rm red\,\,bus}, {\rm blue\,\,bus}, {\rm taxi}\}$.  The decision maker (DM) is indifferent between taking a bus or a taxi and flips a coin to decide between these two modes of travel.  Within the bus category, the DM does not care about the color of the bus and therefore flips a coin to decide on red vs.~blue if both choices are available.  Set ${\cal G} = \{\Omega, \{{\rm blue\,\,bus}, {\rm taxi}\}\}$ and consider the following three PCPSs $p^r$, $p^b$, $p^t$:
\begin{align}
p^r_{\Omega}(\{{\rm red\,\,bus}\}) = 1,&\,\, p^r_{\{{\rm blue\,\,bus}, {\rm taxi}\}}(\{{\rm blue\,\,bus}\}) = 1, ~\label{eq44}\\
p^b_{\Omega}(\{{\rm blue\,\,bus}\}) = 1,&\,\, p^b_{\{{\rm blue\,\,bus}, {\rm taxi}\}}(\{{\rm blue\,\,bus}\}) = 1, ~\label{eq45}\\
p^t_{\Omega}(\{{\rm taxi}\}) = 1,&\,\, p^t_{\{{\rm blue\,\,bus}, {\rm taxi}\}}(\{{\rm taxi}\}) = 1. ~\label{eq46}
\end{align}

The DM puts probability $1/4$ on each of the first and second PCPSs $p^r$ and $p^b$, and probability $1/2$ on the third PCPS $p^t$.  If the choice set is $\Omega$, then the DM averages over the PCPSs to arrive at probability $1/2$ on choosing a bus and probability $1/2$ on choosing a taxi.  If the choice set is $\{{\rm blue\,\,bus}, {\rm taxi}\}$, then the DM averages to get again probability $1/2$ on a bus and probability $1/2$ on a taxi.  This is the intuitive set of choice probabilities.  But, in the first case, the probability ratio of blue bus to taxi is $1:2$, while in the second case it is $1:1$.  This violates (G)IIA.  Since $\cal G$ is a nonempty $\pi$-system, Corollary~\ref{cor-CC} and Theorem~\ref{thm-level-2} tell us that this mixture of PCPSs does not lie in Level 2.
\end{Ex}

As an aside, we observe that the Compromise Effect (Simonson, 1989; Simon and Tversky, 1992) can be represented in exactly the same manner, even though it is often described quite differently in behavioral terms.

\begin{Ex}
Let $\Omega = \{\rm{l\text{-}camera}, \rm{m\text{-}camera}, \rm{h\text{-}camera}\}$, where $\rm{l\text{-}camera}$ is the low-quality/low-price option, $\rm{m\text{-}camera}$ is the medium-quality/medium-price option, and $\rm{h\text{-}camera}$ is the high-quality/high-price option.  Set ${\cal G} = \{\{\rm{l\text{-}camera}, \rm{m\text{-}camera}\}$, $\Omega\}$.  Faced with the choice set $\{\rm{l\text{-}camera}, \rm{m\text{-}camera}\}$, the DM selects each camera with equal probability.  Faced with the choice set $\Omega$, the DM selects each of the $\rm{l\text{-}camera}$ and $\rm{h\text{-}camera}$ with probability $1/4$, and selects the $\rm{m\text{-}camera}$ with probability $1/2$.  This again is a violation of (G)IIA.  The usual story is that the addition of the $\rm{h\text{-}camera}$ to the choice set emphasizes the inferiority of the $\rm{l\text{-}camera}$, which in turn makes the $\rm{m\text{-}camera}$ stand out as a good compromise between the low quality of the $\rm{l\text{-}camera}$ and the high price of the $\rm{h\text{-}camera}$.

This scenario can be represented via a probabilistic mixture of PCPSs: 
\begin{align}
p_{\Omega}^l(\{\rm{l\text{-}camera}\}) = 1,& \,\, p_{\{\rm{l\text{-}camera, m\text{-}camera}\}}^l(\{\rm{l\text{-}camera}\}) = 1, ~\label{eq47} \\
p_{\Omega}^h(\{\rm{h\text{-}camera}\}) = 1,& \,\, p_{\{\rm{l\text{-}camera, m\text{-}camera}\}}^h(\{\rm{l\text{-}camera}\}) = 1, ~\label{eq48} \\
p_{\Omega}^m(\{\rm{m\text{-}camera}\}) = 1,& \,\, p_{\{\rm{l\text{-}camera, m\text{-}camera}\}}^m(\{\rm{m\text{-}camera}\}) = 1. ~\label{eq49}
\end{align}
The DM puts probability $1/4$ on each of the first and second PCPSs $p^{l}$ and $p^{h}$, and probability $1/2$ on the third PCPS $p^{m}$.  The key here is that when the choice set expands to include the $\rm{h\text{-}camera}$, then, with probability $1/4$, the DM will switch from the $\rm{l\text{-}camera}$ to the $\rm{h\text{-}camera}$, reflecting the inferiority of the former.
\end{Ex}

Next, the Attraction Effect (Huber, Payne, and Puto, 1982; Simonson, 1989) and the Repulsion Effect (Aaker, 1991) both violate Regularity.  This establishes that Level 4 strictly nests Level 3 of our hierarchy, since PCPS-coherence implies Regularity.

\begin{proposition} \label{thm-bm-reg}
Fix a measurable space $(\Omega, {\cal F})$ and a subfamily $\mathcal G$ of $\cal F$.  If a stochastic choice rule $G \mapsto c_G(\cdot)$ satisfies PCPS-coherence, then it satisfies Regularity.
\end{proposition}

\begin{proof}
Fix $E \subseteq F \subseteq G$ with $E \in \mathcal F$ and $F, G \in \mathcal G$, and a stochastic choice rule $G \mapsto c_G(\cdot)$ satisfying PCPS-coherence.  We want to show that $c_G(E) \leq c_F(E)$.  
Fix a PCPS $G \mapsto p_G(\cdot)$.  The chain rule gives $p_G(E) = p_G(F) \times p_F(E)$.  It follows that if $p_G(E) = 1$, then both terms on the right side must be $1$.  In particular, we have $p_F(E) = 1$.  From this:
\begin{equation}~\label{77}
[E, G]_\Pi = \{ p \in \Pi : p_G(E) = 1 \} \subseteq \{ p \in \Pi : p_F(E) = 1 \} = [E, F]_\Pi.
\end{equation}
Using Theorem~\ref{thm-level-3}, we write $c_F(E) = Q([E, F]_\Pi)$ and $c_G(E) = Q([E, G]_\Pi)$.  Monotonicity of $Q$ yields $c_G(E) \leq c_F(E)$, as required.
\end{proof}

We now couch the Attraction and Repulsion Effects in terms of signed probabilistic mixtures of PCPSs.

\begin{Ex} \label{ex-b3}
Let $\Omega = \{x, y, z\}$, ${\cal G} = \{\Omega, \{x, y\}\}$, and consider a stochastic choice rule satisfying:
\begin{equation} \label{eq83}
c_{\{x, y\}}(\{x\}) < c_\Omega(\{x\}),
\end{equation}
which is a violation of Regularity.  In one scenario for the Attraction Effect (Simonson and Tversky, 1992), item $x$ is a nice pen, item $y$ is a certain sum of money, and item $z$ is a plain pen.  The addition of item $z$ highlights the attractiveness of item $x$, which is then chosen with higher probability.  In a scenario for the Repulsion Effect (Simonson, 2014; Kruis et al., 2020), item $x$ is candy, item $y$ is an orange, and item $z$ is a spoiled clementine.  The addition of item $z$ casts doubt on the freshness of item $y$, so that, again, item $x$ is chosen with higher probability.

We now show how to produce these effects via a signed probabilistic mixture of PCPSs.  Specifically, consider four PCPSs relative to $\cal G$:
\begin{align}
p^1_{\{x,y\}}(\{x\}) = 1,& \,\, p^1_{\{x,y,z\}}(\{x\}) = 1, ~\label{eq89} \\
p^2_{\{x,y\}}(\{y\}) = 1,& \,\, p^2_{\{x,y,z\}}(\{y\}) = 1, ~\label{eq90} \\
p^3_{\{x,y\}}(\{x\}) = 1,& \,\, p^3_{\{x,y,z\}}(\{z\}) = 1, ~\label{eq91}\\
p^4_{\{x,y\}}(\{y\}) = 1,& \,\, p^4_{\{x,y,z\}}(\{z\}) = 1. ~\label{eq92}
\end{align}
We put a signed probability measure $(q^1, q^2, q^3, q^4)$, where each $q^i \in \Re$ and $\sum_{i=1}^4 q^i = 1$, on these four PCPSs and obtain:
\begin{align}
c_{\{x,y\}}(\{x\}) &= q^1 + q^3, ~\label{eq93} \\
c_{\{x,y,z\}}(\{x\}) &= q^1. ~\label{eq94}
\end{align}

In order to create a violation of Regularity, we must have $q^3 < 0$.  For example, setting $(q^1, q^2, q^3, q^4) = (3/5, 1/5, -1/10, 3/10)$ works.  This makes intuitive sense.  The PCPS $p^3$ selects item $x$ -- the nice pen or the candy -- from the choice set $\{x, y\}$.  But it selects item $z$ -- the plain pen or spoiled clementine -- from the larger set $\{x, y, z\}$.  We expect the DM to want to avoid choosing according to this PCPS.  (We do not rule out that $q^4 < 0$ as well.) 

Our analysis here is similar to the treatment of the Attraction Effect by Dogan and Yildiz (2023, Example 2) in terms of orders.  A story to go with the example is that the DM is a principal and there are four agents, each with one of the four PCPSs above.  (This might be a multiple-selves story.)  The principal has preferences over which agent gets to make choices and dislikes the event that agent 3 -- with PCPS $p^3$ -- is the one to choose.  Formally, the DM has a negative willingness-to-bet on this event, which is an interpretation of a negative subjective probability from Brandenburger et al., 2026).
\end{Ex}

\section*{Declaration of Generative AI and AI-Assisted Technologies in the Writing Process}

During the preparation of this work the authors used ChatGPT for proof tactics and checking, copy-editing, and bibliographic search.  After using this tool, the authors reviewed and edited the content as needed and take full responsibility for the content of the published article.

\end{document}